\documentclass[a4paper,10pt]{amsart}
\usepackage[english]{babel}
\usepackage{amsmath,amsthm,amssymb,amsfonts}
\usepackage{multicol}
\usepackage[titletoc]{appendix}

\usepackage[pdftex]{color}
\usepackage[bookmarks=true,hyperindex,pdftex,colorlinks, citecolor=blue,linkcolor=blue, urlcolor=blue]{hyperref}
\usepackage[dvipsnames]{xcolor}
\usepackage{mathtools}
\usepackage{graphicx}

\usepackage[normalem]{ulem}

\theoremstyle{definition}
\newtheorem{theorem}{Theorem}[section]
\newtheorem{cor}[theorem]{Corollary}
\newtheorem{prop}[theorem]{Proposition}
\newtheorem{lemma}[theorem]{Lemma}
\newtheorem{rem}[theorem]{Remark}

\newtheorem{definition}[theorem]{Definition}

\newcommand{\R}{\mathbb{R}}
\newcommand{\N}{\mathbb{N}}

\newcommand{\C}{\mathbb{C}}
\newcommand{\K}{\mathbb{K}}

\newcommand{\e}{\varepsilon}
\newcommand{\eps}{\varepsilon}

\newcommand{\tensor}{{\hat \otimes}_{\pi}}
\newcommand{\tens}{\otimes}

\DeclareMathOperator{\re}{Re}
\DeclareMathOperator{\im}{Im}

\DeclareMathOperator{\NA}{NA}

\renewcommand{\leq}{\leqslant}
\renewcommand{\geq}{\geqslant}

\renewcommand{\geq}{\geqslant}

\begin{document}

\title{Strong subdifferentiability and local Bishop-Phelps-Bollob\'as properties}

\author[Dantas]{Sheldon Dantas}
\address[Dantas]{Department of Mathematics, Faculty of Electrical Engineering, Czech Technical University in Prague, Technick\'a 2, 166 27 Prague 6, Czech Republic \newline
	\href{http://orcid.org/0000-0001-8117-3760}{ORCID: \texttt{0000-0001-8117-3760}  }}
\email{\texttt{gildashe@fel.cvut.cz}}

\author[Kim]{Sun Kwang Kim}
\address[Kim]{Department of Mathematics, Chungbuk National University, 1 Chungdae-ro, Seowon-Gu, Cheongju, Chungbuk 28644, Republic of Korea \newline
	\href{http://orcid.org/0000-0002-9402-2002}{ORCID: \texttt{0000-0002-9402-2002}  }}
\email{\texttt{skk@chungbuk.ac.kr}}

\author[Lee]{Han Ju Lee}
\address[Lee]{Department of Mathematics Education, Dongguk University - Seoul, 04620 (Seoul), Republic of Korea \newline
	\href{http://orcid.org/0000-0001-9523-2987}{ORCID: \texttt{0000-0001-9523-2987}  }
}
\email{\texttt{hanjulee@dongguk.edu}}

\author[Mazzitelli]{Martin Mazzitelli}
\address[Mazzitelli]{Universidad Nacional del Comahue, CONICET, Departamento de Matem\'atica, Facultad de Econom\'ia y Administraci\'on, Neuqu\'en, Argentina.}
\email{\texttt{martin.mazzitelli@crub.uncoma.edu.ar}}

\begin{abstract} It has been recently presented in \cite{DKLM} some local versions of the Bishop-Phelps-Bollob\'as type property for operators. In the present article, we continue studying these properties for multilinear mappings. We show some differences between the local and uniform versions of the Bishop-Phelps-Bollob\'as type results for multilinear mappings, and also provide some interesting examples which shows that this study is not just a mere generalization of the linear case. We study those properties for bilinear forms on $\ell_p \times \ell_q$ using the strong subdifferentiability of the norm of the Banach space $\ell_p \hat{\tens}_{\pi} \ell_{q}$. Moreover, we present necessary and sufficient conditions for the norm of a Banach space $Y$ to be strongly subdifferentiable through the study of these properties for bilinear mappings on $\ell_1^N \times Y$.

\end{abstract}

\thanks{The first author was supported by the project OPVVV CAAS CZ.02.1.01/0.0/0.0/16\_019/0000778. The second author was partially supported by Basic Science Research Program through the National Research Foundation of Korea(NRF) funded by the Ministry of Education, Science and Technology (NRF-2017R1C1B1002928). The third author was supported by the Research program through the National Research Foundation of Korea funded by the Ministry of Education, Science and Technology (NRF-2016R1D1A1B03934771). The fourth author was partially supported by CONICET PIP 11220130100329CO}


\maketitle

\thispagestyle{plain}

\section{Introduction}

In Banach space theory, it is well-known that the set of all norm attaining continuous linear functionals defined on a Banach space $X$ is dense in its topological dual space $X^*$. This is the famous Bishop-Phelps theorem \cite{BP}. In 1970, this result was strengthened by Bollob\'as, who proved a quantitative version in the following sense: if a norm-one linear functional $x^*$ almost attains its norm at some $x$, then, near to $x^*$ and $x$, there are, respectively, a new norm-one functional $y^*$ and a new point $y$ such that $y^*$ attains its norm at $y$ (see \cite[Theorem 1]{Bol}). Nowadays, this result is known as the Bishop-Phelps-Bollob\'as theorem and it has been used as an important tool in the study of Banach spaces and operators. For example, it was used to prove that the numerical radius of a continuous linear operator is the same as its adjoint.

It is natural to ask whether the Bishop-Phelps and Bishop-Phelps-Bollob\'as theorems hold also for bounded linear operators. In 1963, Lindenstrauss gave the first example of a Banach space $X$ such that the set of all norm attaining operators on $X$ is {\it not} dense in the set of all bounded linear operators (see \cite[Proposition 5]{Lind}). On the other hand, he studied some conditions on the involved Banach spaces in order to get a Bishop-Phelps type theorem for operators. For instance, he proved that the set of all operators whose second adjoint attain their norms is dense, so, in particular, if $X$ is a reflexive Banach space, then the set of all norm attaining operators is dense for arbitrary range spaces (actually, this result was extended by Bourgain in \cite[Theorem 7]{Bou} by showing that the Radon-Nikod\'ym property implies the same result). This topic has been considered by many authors and we refer the reader to the survey paper \cite{Acosta-RACSAM} for more information and background about denseness of norm attaining operators. On the other hand, M. Acosta, R. Aron, D. Garc\'{\i}a, and M. Maestre studied the vector-valued case of the Bishop-Phelps-Bollobas theorem and introduced \cite{AAGM} the {\it Bishop-Phelps-Bollob\'{a}s property}. 

Now we introduce the notation and necessary preliminaries. Let $N$ be a natural number. We use capital letters $X, X_1, \ldots, X_N, Y$ for Banach spaces over a scalar field $\K$ which can be the field of the real numbers $\R$ or the field of the complex numbers $\C$. The closed unit ball and the unit sphere of $X$ are denoted by $B_X$ and $S_X$, respectively. The topological dual space of $X$ is denoted by $X^*$ and $\mathcal{L} (X_1, \ldots, X_N; Y)$ stands for the set of all bounded $N$-linear mappings from $X_1 \times \cdots \times X_N$ into $Y$. For the convenience, if $X_1 = \ldots = X_N = X$, then we use the shortened notation $\mathcal{L} (^N X; Y)$. When $N = 1$, we have the set of all bounded linear operators from $X$ into $Y$, which we denote simply by $\mathcal{L}(X; Y)$.


We say that an $N$-linear mapping $A \in \mathcal{L} (X_1, \ldots, X_N; Y)$ {\it attains its norm} if there exists $(z_1 ,\ldots, z_N ) \in S_{X_1} \times \cdots \times S_{X_N}$ such that $\|A(z_1 , \ldots, z_N )\| =\|A\|$, where $\|A\| = \sup \|A(x_1, \ldots, x_N)\|$, the supremum being taken over all the elements $(x_1, \ldots, x_N) \in S_{X_1} \times \cdots \times S_{X_N}$. We denote by $\NA(X_1, \ldots, X_N; Y)$ the set of all norm attaining $N$-linear mappings.

\begin{definition}[\cite{ABGM, CS, DGKLM, KLM}] We say that $(X_1, \ldots,  X_N; Y)$ has the {\it Bishop-Phelps-Bollob\'as property for $N$-linear mappings} ({\bf BPBp} for $N$-linear mappings, for short) if given $\e > 0$, there exists $\eta(\e) > 0$ such that whenever $A \in \mathcal{L} (X_1, \ldots, X_N; Y)$ with $\|A\| = 1$ and $\left(x_1, \ldots, x_N\right) \in S_{X_1} \times \cdots \times S_{X_N}$ satisfy
\begin{equation*}
\left\|A\left(x_1 , \ldots, x_N \right)\right\| > 1 - \eta(\e),
\end{equation*}
there are $B \in \mathcal{L} (X_1, \ldots, X_N; Y)$ with $\|B\| = 1$ and $\left(z_1, \ldots, z_N\right) \in S_{X_1} \times \cdots \times S_{X_N}$ such that
\begin{equation*}
\left\|B\left(z_1 , \ldots, z_N \right)\right\| = 1, \ \ \ \max_{1 \leq j \leq N} \|z_j  - x_j \| < \e, \ \ \ \mbox{and} \ \ \|B - A\| < \e.
\end{equation*}
\end{definition}
\noindent
When $N = 1$, we simply say that the pair $(X; Y)$ satisfies the {\bf BPBp} (see \cite[Definition 1.1]{AAGM}). Note that the Bishop-Phelps-Bollob\'as theorem asserts that the pair $(X; \K)$ has the {\bf BPBp} for every Banach space $X$. It is immediate to notice that if the pair $(X; Y)$ has {\bf BPBp}, then $\overline{\NA(X;Y)}=\mathcal{L}(X;Y)$. However, the converse is {\it not} true even for finite dimensional spaces. Indeed, for a finite dimensional Banach space $X$, the fact that $B_X$ is compact implies that every bounded linear operator on $X$ attains its norm, but it is known that there is some Banach space $Y_0$ so that the pair $(\ell_1^2; Y_0)$ {\it fails} the {\bf BPBp} (see \cite[Example 4.1]{ACKLM}). This shows that the study of the {\bf BPBp} is not just a trivial extension of that of the density of norm attaining operators.

Similar to the case of operators, there were a lot of attention to the study of the denseness of norm attaining bilinear mappings. It was proved that, in general, there is no Bishop-Phelps theorem for bilinear mappings (see \cite[Corollary 4]{AAP}). Moreover, it is known that $\overline{\NA(^2 L^1[0,1];\mathbb{K})} \neq \mathcal{L}(^2 L^1[0,1];\mathbb{K})$ (see \cite[Theorem 3]{C}), even though $\overline{\NA(L^1[0,1];L_\infty[0,1])} = \mathcal{L}(L^1[0,1];L_\infty[0,1])$ (see \cite{FinetPaya}). This result is interesting since the Banach space $\mathcal{L}(X_1,X_2;\mathbb{K})$ is isometrically isomorphic to $\mathcal{L}(X_1;X_2^*)$ via the canonical isometry $A\in \mathcal{L}(X_1,X_2;\mathbb{K})\longmapsto T_A \in \mathcal{L}(X_1;X_2^*)$ given by $[T_A(x_1)](x_2) = A(x_1,x_2)$. Concerning the {\bf BPBp} for bilinear mappings, it is known that $(\ell_1,\ell_1;\mathbb{K})$ fails the {\bf BPBp} for bilinear mappings (see \cite{CS}) but the pair $(\ell_1, Y)$ satisfies the {\bf BPBp} for many Banach spaces $Y$, including $\ell_\infty$ (see \cite[Section 4]{AAGM}). We refer the papers \cite{ABGM, DGKLM, KLM} for more results on the {\bf BPBp} for multilinear mappings.


Very recently, a stronger property than the {\bf BPBp} was defined and studied.
\begin{definition}[\cite{DKL, DKKLM}]
We say that $(X_1, \ldots,  X_N; Y)$ has the {\it Bishop-Phelps-Bollob\'as point property for $N$-linear mappings} ({\bf BPBpp} for $N$-linear mappings, for short) if given $\e > 0$, there exists $\eta(\e) > 0$ such that whenever $A \in \mathcal{L} (X_1, \ldots, X_N; Y)$ with $\|A\| = 1$ and $\left(x_1, \ldots, x_N\right) \in S_{X_1} \times \cdots \times S_{X_N}$ satisfy
\begin{equation*}
\left\|A\left(x_1 , \ldots, x_N \right)\right\| > 1 - \eta(\e),
\end{equation*}
there is $B \in \mathcal{L} (X_1, \ldots, X_N; Y)$ with $\|B\| = 1$ such that
\begin{equation*}
\left\|B\left(x_1 , \ldots, x_N \right)\right\| = 1 \ \ \ \mbox{and} \ \ \  \|B - A\| < \e.
\end{equation*}
\end{definition}
\noindent Clearly, the {\bf BPBpp} implies the {\bf BPBp} but the converse is {\it not} true in general. Actually, if the pair $(X; Y)$ has the {\bf BPBpp} for some Banach space $Y$, then $X$ must be uniformly smooth (see \cite[Proposition 2.3]{DKL}). Also, it was proved in \cite{DKL} that the pair $(X; \K)$ has the {\bf BPBpp} if and only if $X$ is uniformly smooth. In both papers \cite{DKL, DKKLM} the authors presented such differences between these two properties and found many positive examples having {\bf BPBpp}.

On the other hand, one may think about a ``dual" version of the {\bf BPBpp} where, instead of fixing the point, we fix the operator.
\begin{definition}[\cite{D, DKLM}]
We say that $(X_1, \ldots,  X_N; Y)$ has the {\it Bishop-Phelps-Bollob\'as operator property for $N$-linear mappings} ({\bf BPBop} for $N$-linear mappings, for short) if given $\e > 0$, there exists $\eta(\e) > 0$ such that whenever $A \in \mathcal{L} (X_1, \ldots, X_N; Y)$ with $\|A\| = 1$ and $\left(x_1, \ldots, x_N\right) \in S_{X_1} \times \cdots \times S_{X_N}$ satisfy
\begin{equation*}
\left\|A\left(x_1 , \ldots, x_N \right)\right\| > 1 - \eta(\e),
\end{equation*}
there is $(z_1,\dots, z_N) \in S_{X_1}\times \cdots\times S_{X_N}$ such that
\begin{equation*}
\left\|A\left(z_1 , \ldots, z_N \right)\right\| = 1 \ \ \ \mbox{and} \ \ \   \max_{1 \leq j \leq N}  \|x_j - z_j\| < \e.
\end{equation*}
\end{definition}

It was proved in \cite{KL} that the pair $(X; \K)$ has the {\bf BPBop} if and only $X$ is uniformly convex. So, in the scalar-valued case, these two properties are dual from each other; that is, $(X; \K)$ has the {\bf BPBpp} if and only if $(X^*; \K)$ has the {\bf BPBop}. Nevertheless, it is known that there is {\it no} version for bounded linear operators of this property. Indeed, in \cite{DKKLM2}, it is proved that for $\dim(X), \dim(Y) \geq 2$, the pair $(X; Y)$ {\it always  fails} the {\bf BPBop}. Hence, there is no hope for this ``uniform" property, which lead us to consider a ``local type" of it as in \cite{D, Sain, T}. In these papers, the function $\eta$ in the definition of the {\bf BPBop} depends not only on $\e$ but also on a fixed norm one operator $T$, and some positive results are obtained, which are different from the uniform case when $\eta$ depends just on $\e$.

This motivated the current authors to study, in \cite{DKLM}, all of the aforementioned properties in this local sense. In the paper, local versions of the {\bf BPBpp} and {\bf BPBop} (and also the {\bf BPBp}) were addressed for linear operators. We give the precise definitions for $n$-linear mappings in section 2.  It turned out that these local properties are quite different from the corresponding uniform ones, as in the case of the {\bf BPBop} (see \cite[Section 5]{DKLM}). For instance, there is a connection between those properties and the subdifferentiability of the norm of the spaces (see \cite[Theorem 2.3]{DKLM}). For the ``local {\bf BPBpp}", $\eta$ depends on a point $x \in S_X$ and $\e > 0$, we have the following results.

\begin{theorem}[\cite{DKLM}] \label{resume} Consider the following pairs of Banach spaces $(X; \K)$ when $X$ is
\begin{itemize}
	\item[(a)] $c_0$ or
	\item[(b)] the predual of Lorentz sequence space $d_{*}(w, 1)$ or
	\item[(c)] the space $VMO$ (which is the predual of the Hardy space $H^1$) or
	\item[(d)] a finite dimensional space,
\end{itemize}
and also the following pairs
\begin{itemize}
	\item[(e)] $(\ell_1^N; L_p(\mu))$ for $1 < p < \infty, N \in \N$, and
	\item[(f)] $(c_0; L_p(\mu))$ for $1 \leq p < \infty$.
\end{itemize}
Then, all of them satisfy this ``local {\bf BPBpp}".
\end{theorem}
In this paper we continue the study of these local properties, emphasizing in the multilinear setting. Following the notation  in \cite{DKLM}, we use the symbol {\bf L}$_{p, p}$ for the ``local {\bf BPBpp}", when $\eta$ depends on a point $x \in S_X$, and {\bf L}$_{o, o}$ for the ``local {\bf BPBop}", when $\eta$ depends on an operator $T \in S_{\mathcal{L}(X, Y)}$ (see Definition \ref{defs} below). In the next section, we give the proper definitions and first results. Among others, we obtain the following results (see Proposition~\ref{bil3} and the comment below Corollary~\ref{failBPBop}).
\begin{itemize}
\item  If $(X_1, \ldots ,X_N;Y)$ has property {\bf L}$_{p, p}$ (or {\bf L}$_{o, o}$), then so does $(X_i;\mathbb{K})$ for every $1\leq i\leq N$.
\item There exist (finite dimensional) Banach spaces $X_1,\dots, X_N, Y$ such that $(X_1,  \ldots, X_N;Y)$ has the {\bf L}$_{p, p}$ (respectively, {\bf L}$_{o,o}$) but fails the {\bf BPBpp} (respectively, {\bf BPBop}).
\end{itemize}
We also focus on the bilinear case when the domains are $\ell_p$-spaces. In that sense, we obtain the following results (see Theorem~\ref{lp} and Remark~\ref{remark lp}).
\begin{itemize}
\item If  $2 <p, q < \infty$, then $(\ell_p, \ell_q; \K)$ has the {\bf L}$_{p, p}$.
\item If $1<p,q<\infty$, then $(\ell_p, \ell_q; \mathbb{K})$ has the {\bf L}$_{o,o}$ if and only if $pq>p+q$. Hence, there exist spaces $\ell_p, \ell_q$ such that $(\ell_p,\ell_q;\K)$ fails the bilinear {\bf L}$_{o, o}$, while $(\ell_p;\K)$ and $(\ell_q;\K)$ have the linear {\bf L}$_{o, o}$, since both are uniformly smooth. 
\end{itemize}

In the proof of Theorem~\ref{lp} we use a tensor product to prove that $(\ell_p, \ell_q; \K)$ has the {\bf L}$_{p, p}$ for $2 <p, q < \infty$. As a consequence, we show that the norm of $\ell_p \hat{\tens}_{\pi} \ell_q$ is strongly subdifferentiable for $2 < p,q < \infty$. However if $p^{-1}+q^{-1} \geq 1$ or one of the indices $p,q$ takes the value 1 or $\infty$, then its norm is not strongly subdifferentiable.
In Section 3, motivated by the geometric property \emph{approximate hyperplane series property} (AHSP, for short) in \cite{AAGM, ABGM}, we get a characterization of  strong subdifferentiability. The AHSP characterizes a Banach space $Y$ for which $(\ell_1;Y)$ and $(\ell_1, Y;\K)$ have the {\bf BPBp} (in the linear and bilinear case, respectively). Although the pairs $(\ell_1;Y)$ and $(\ell_1, Y;\K)$ do not have the {\bf L}$_{p, p}$ (since $\ell_1$ is not SSD), we may ask if $(\ell_1^N;Y)$ and $(\ell_1^N, Y;\K)$ have it.
In Proposition~\ref{SSDe} we prove that the strong subdifferentiability of the norm of a Banach space $Y$ is equivalent to such characterization. As a consequence of this characterization, we prove that $(\ell_1^N, Y; \K)$ has the {\bf L}$_{p, p}$ for bilinear forms if and only if the norm of a Banach space $Y$ is strongly subdifferentiable. Using similar ideas, we characterize the pairs $(\ell_1^N; Y)$ having the  {\bf L}$_{p, p}$ for operators, generalizing Theorem~\ref{resume}.(e). As a consequence of this last characterization, we prove that if a family $\{y_{\alpha} \}_{\alpha} \subset S_Y$ is uniformly strongly exposed with corresponding functionals $\{f_{\alpha}\}_{\alpha} \subset S_{Y^*}$, then $(\ell_1^N; Y)$ has the {\bf L}$_{p, p}$ for operators whenever $\{f_{\alpha}\}_{\alpha}$  is a norming subset for the Banach space $Y$.


\section{The {\bf L}$_{p,p}$ and the {\bf L}$_{o,o}$ for $N$-linear mappings}

We start this section by giving the precise definitions of the local Bishop-Phelps-Bollob\'as properties for $N$-linear mappings. These are the analogous of \cite[Definition 2.1]{DKLM}.

\begin{definition}	\label{defs} {\bf (a)} We say that  $(X_1, \ldots,  X_N; Y)$ has the {\bf L}$_{p, p}$ if given $\e > 0$ and $(x_1,\dots ,x_N) \in S_{X_1} \times \cdots \times S_{X_N}$, there is $\eta(\e, (x_1,\dots,x_N)) > 0$ such that whenever $A \in \mathcal{L}( X_1, \ldots,  X_N; Y)$ with $\|A\| = 1$ satisfies
	\begin{equation*}
		\|A(x_1,\dots,x_N)\| > 1 - \eta(\e, (x_1,\dots,x_N)),
	\end{equation*}
	there is $B \in \mathcal{L}(X_1,\dots,X_N;Y)$ with $\|B\| = 1$ such that
	\begin{equation*}
		\|B(x_1,\dots,x_N)\| = 1 \ \ \ \mbox{and} \ \ \ \|B - A\| < \e.	
	\end{equation*}
\noindent
{\bf (b)} We say that $(X_1,\dots,X_N;Y)$ has the {\bf L}$_{o,o}$ if given $\e > 0$ and $A \in \mathcal{L}(X_1,\dots,X_N;Y)$ with $\|A\| = 1$, there is $\eta(\e, A) > 0$ such that whenever $(x_1,\dots,x_N) \in S_{X_1} \times \cdots \times S_{X_N}$ satisfies
	\begin{equation*}
		\|A(x_1,\dots,x_N)\| > 1 - \eta(\e, A),
	\end{equation*}
	there is $(z_1,\dots,z_N) \in S_{X_1} \times \cdots \times S_{X_N}$ such that
	\begin{equation*}
		\|A(z_1,\dots,z_N)\| = 1 \ \ \mbox{and} \ \ \max_{1 \leq j \leq N}  \|x_j - z_j\| < \e.	
	\end{equation*}
\end{definition}

Let us observe that if $(X_1, \ldots,  X_N; Y)$  satisfies the {\bf L}$_{o,o}$, then every $A \in \mathcal{L}(X_1,\dots,X_N;Y)$ attains its norm and, consequently, all the Banach spaces $X_i$'s must be reflexive. Indeed, if one of them is not reflexive, say $X_k$, by James theorem, there is $z_k^* \in S_{X_k^*}$ such that $|z_k^*(x_k)| < 1$ for all $x_k \in S_{X_k}$. Now, taking arbitrary $y_0\in S_Y$ and $z_i^* \in S_{X_i^*}$ for each $i\neq k$ and defining $A \in \mathcal{L}(X_1,...,X_N;Y)$ by $A(x_1,...,x_N) := \left(\Pi_{1\leq i \leq N} z_i^*(x_i)\right) y_0$, we see that $A$ never attains its norm. Thus, in order to look for positive examples about the {\bf L}$_{o,o}$, we must assume, at least, that $X_1, \ldots, X_N$ are all reflexive Banach spaces.

It was proved in \cite[Theorem 2.4]{D} that if $X$ is a finite dimensional Banach space, then the pair $(X; Y)$ has the {\bf L}$_{o,o}$ for every Banach space $Y$. By the similar proof, this can be generalized for $N$-linear mappings. However it does not hold  for the {\bf L}$_{p,p}$ in general. Indeed, suppose that $Y$ is a  strictly convex Banach space and that the pair $(\ell_1^2; Y)$ has the {\bf L}$_{p,p}$. 
Then $Y$ is uniformly convex by \cite[Proposition 3.2]{DKLM}. So,  choosing a strictly convex space  $Y_0$ which is not uniformly convex, the pair $(\ell_1^2, Y_0)$ fails the {\bf L}$_{p,p}$ although $\ell_1^2$ is $2$-dimensional. In the case that  $Y$ is also finite dimensional, then we have a positive result as the following proposition. The proof is analogous to the operator case  in  \cite[Proposition 2.8]{DKLM} and omitted. 
\begin{prop} \label{fifi}	
Let $N \in \N$ and let $X_1, \ldots, X_N$ be finite dimensional Banach spaces. Then,
\begin{enumerate}
\item[(a)] $(X_1,\dots,X_N;Y)$ has the {\bf L}$_{o,o}$ for every Banach space $Y$;
\item[(b)] $(X_1,\dots,X_N;Y)$ has the {\bf L}$_{p,p}$ for every finite dimensional Banach space $Y$.
\end{enumerate}
\end{prop}
It is known that if the pair $(X; Y)$ satisfies the {\bf BPBpp} or the {\bf BPBop} or the {\bf L}$_{p, p}$ for some Banach space $Y$, then so does $(X; \K)$ (see \cite[Proposition 2.9]{D}, \cite[Proposition 2.7]{DKL} and \cite[Proposition 2.7]{DKLM}, respectively). The same happens with property {\bf L}$_{o, o}$. Indeed, given $\e > 0$ and $x^* \in S_{X^*}$, we construct, for a fixed $y_0 \in S_Y$, the operator $T \in \mathcal{L}(X, Y)$ given by $T(x) := x^*(x)y_0$ for all $x \in X$ and then we set $\eta(\e, x^*) := \eta(\e, T) > 0$. If $x_0 \in S_X$ is such that
\begin{equation*}
|x^*(x_0)| > 1 - \eta(\e, x^*),
\end{equation*}
then $\|T(x_0)\| > 1 - \eta(\e, T)$. Thus, there is $x_1 \in S_X$ such that
\begin{equation*}
\|T(x_1)\| = |x^*(x_1)| = 1 \ \ \mbox{and} \ \  \|x_1 - x_0\| < \e.
\end{equation*}
Therefore, $(X, \K)$ has the {\bf L}$_{o, o}$. By using the same arguments, we can extend those results for $N$-linear mappings. In the proof, we use the canonical isometry between $\mathcal{L}(X_1,...,X_N;\mathbb{K})$ and $\mathcal{L}(X_1,...,X_{N-1};X_N^*)$ to deduce item (b) below.
\begin{prop}\label{bil3} Let $\mathcal{P}$ be one of the properties {\bf BPBpp}, {\bf BPBop}, {\bf L}$_{o,o}$ or {\bf L}$_{p,p}$.
	\begin{enumerate}
		\item[(a)] If $(X_1, \ldots ,X_N;Y)$ has the property $\mathcal{P}$, then so does $(X_1, \ldots ,X_N;\mathbb{K})$.
		\item[(b)] If $(X_1, \ldots ,X_N;\mathbb{K})$ has the property $\mathcal{P}$ and $\mathcal{P}$ is not  {\bf L}$_{p, p}$, then so does $(X_1, \ldots ,X_{N-1};X_N^*)$.
		\item[(c)]  If $(X_1, \ldots ,X_N;Y)$ has the property $\mathcal{P}$, then so does $(X_i;\mathbb{K})$ for every $1\leq i\leq N$.
	\end{enumerate}
\end{prop}
\begin{proof}
The proof of (a) and (b) is sketched above.
To prove (c), it suffices to show that the pair $(X_1;\mathbb{K})$ has property $\mathcal{P}$ whenever $(X_1, \ldots ,X_N;Y)$ does. Suppose first that $\mathcal{P}$ is not  {\bf L}$_{p, p}$. Then, by item (a), we have that $(X_1, \ldots ,X_N;\mathbb{K})$ has property $\mathcal{P}$ and, in virtue of (b), $(X_1, \ldots ,X_{N-1};X_N^*)$ does. Applying (a) again, we see that $(X_1, \ldots ,X_{N-1};\mathbb{K})$ has property $\mathcal{P}$. That is, if $(X_1, \ldots ,X_N;\mathbb{K})$ has property $\mathcal{P}$, then $(X_1, \ldots ,X_{N-1};\mathbb{K})$ has property $\mathcal{P}$. Repeating this argument $(N-1)$-times, we see that $(X_1;\mathbb{K})$ has property $\mathcal{P}$.

Now, suppose that $(X_1, \ldots ,X_N;Y)$ has property {\bf L}$_{p, p}$. Then, by (a), we have that $(X_1, \ldots ,X_N;\mathbb{K})$ has property {\bf L}$_{p, p}$. Given $\e>0$ and $x_1^0\in S_{X_1}$, we want to see that there is $\eta(\e,x_1^0)>0$ satisfying the definition of property {\bf L}$_{p, p}$ for the pair $(X_1; \mathbb{K})$. Consider $(x_2^0,\dots,x_N^0)\in S_{X_2}\times \cdots \times S_{X_N}$ and $(x_2^*,\dots,x_N^*)\in S_{X_2^*}\times \cdots \times S_{X_N^*}$ such that $x_i^*(x_i^0)=1$, for $i=2, \ldots, N$, and put $\eta(\e,x_1^0):=\eta(\e,(x_1^0,\dots, x_N^0))$, which exists by hypothesis. Suppose that $x_1^*\in S_{X_1^*}$ is such that $|x_1^*(x_1^0)|>1-\eta(\e,x_1^0)$. Then, defining $A(x_1,\dots, x_N)=x_1^*(x_1)x_2^*(x_2)\cdots x_N^*(x_N)$, we have that  $A \in \mathcal{L}( X_1, \ldots,  X_N; \mathbb{K})$, $\|A\|=1$, and
$$
|A(x_1^0,\dots, x_N^0)|>1-\eta(\e,(x_1^0,\dots,x_N^0)).
$$
Consequently, there exists $B \in \mathcal{L}(X_1,...,X_N;\mathbb{K})$ with $\|B\| = 1$ such that $|B(x_1^0,...,x_N^0)| = 1$ and $\|B - A\| < \e$. Therefore, defining $y_1^*\in X_1^*$ by $y_1^*(\cdot)=B(\cdot, x_2^0,\dots, x_N^0)$, we see that
\begin{equation*}
y_1^* \in S_{X_1^*}, \ \ \ |y_1^*(x_1^0)|=1, \ \ \ \mbox{and} \ \ \ \|y_1^*-x_1^*\|\leq  \|B - A\| < \e,
\end{equation*}
which is the desired statement.
\end{proof}
The item (b) above {\it does not} hold for the {\bf L}$_{p, p}$; we provide a counterexample in Remark~\ref{counter}.

Recall that the norm of a Banach space $X$ is said to be {\it strongly subdifferentiable} (SSD, for short) at $x \in S_X$ if the one-sided limit
\begin{equation}\label{def SSD}
\lim_{t \rightarrow 0^+} \frac{\|x + th\| - \|x\|}{t}
\end{equation}
exists uniformly for $h\in B_X$. If \eqref{def SSD} holds for every element in the unit sphere $S_X$, we say that the norm of $X$ is SSD or just $X$ is SSD.
This differentiability is known to be strictly weaker than Fr\'echet differentiability. By the characterization of SSD due to C. Franchetti and R. Pay\'a (see \cite[Theorem 1.2]{FP}), we have that
$(X; \K)$ has the {\bf L}$_{p,p}$ if and only if the norm of $X$ is SSD and, by duality, $(X; \K)$ has the {\bf L}$_{o,o}$ if and only if $X$ is reflexive and the norm of $X^*$ is SSD (see \cite[Theorem 2.3]{DKLM} and also \cite{GMZ} where this fact was already observed).

By using this result and the characterization of property {\bf BPBpp} for the pair $(X;\K)$ given in \cite[Proposition 2.1]{DKL}, we have the following consequences of Proposition \ref{bil3}.
\begin{cor}\label{coro1} Let $N \in \N$ and $X_1, \ldots, X_N$ be Banach spaces.
	\begin{itemize}
	\item[(a)] If $(X_1, \ldots, X_N; Y)$ has the {\bf BPBpp} for some Banach space $Y$, then $X_i$ is uniformly smooth for each $i=1,\ldots,N$.
	\item[(b)] If $(X_1, \ldots, X_N;Y)$ has the {\bf L}$_{p,p}$ for some Banach space $Y$, then $X_i$ is SSD for each $i=1,\ldots,N$.
	\end{itemize}	
\end{cor}
Another consequence of Proposition \ref{bil3} is that, for spaces of dimension greater than 2, there is no {\bf BPBop} for bilinear mappings. Indeed, if $\dim(X), \dim(Y) \geq 2$ and $(X, Y; Z)$ has the {\bf BPBop} for some Banach space $Z$, then by Proposition \ref{bil3}, the pair $(X, Y^*)$ has the {\bf BPBop} for operators and, as we already mentioned in the Introduction, this is not possible. We can deduce the same for $N$-linear mappings.

\begin{cor}\label{failBPBop} Let $N \in \N$. Let $X_i$ be a Banach space with $\dim(X_i) \geq 2$ for $1 \leq i \leq N$. Then, $(X_1, \ldots, X_N; Y)$ fails the {\bf BPBop} for every Banach space $Y$.
\end{cor}
 At this point, we can point out some differences between properties {\bf BPBpp} (respectively, {\bf BPBop}) and {\bf L}$_{p,p}$ (respectively, {\bf L}$_{o,o}$). For instance, if $X_i=\ell_1^2$ or $\ell_\infty^2$ and $Y$ is any finite dimensional Banach space, then by Proposition~\ref{fifi} we have that $(X_1,  \cdots, X_N;Y)$ has the {\bf L}$_{p, p}$ (respectively, {\bf L}$_{o,o}$) while, in virtue of Corollary~\ref{coro1}.(a) (respectively, Corollary~\ref{failBPBop}) it fails property {\bf BPBpp} (respectively, {\bf BPBop}).

Next we focus on the bilinear case when the domains are $\ell_p$-spaces. For the part (b) of Theorem~\ref{lp} below we need the following lemma, which gives a converse of Proposition~\ref{bil3} (b) for property {\bf L}$_{o, o}$.
\begin{lemma} \label{uv} Let $X, Y$ be Banach spaces and suppose that $Y$ is uniformly convex. Then $(X, Y; \K)$ has the {\bf L}$_{o, o}$ for bilinear forms if and only if the pair $(X; Y^*)$ has the {\bf L}$_{o, o}$ for operators.
\end{lemma}
\begin{proof} From Proposition \ref{bil3} (b), if  $(X, Y; \K)$ has the {\bf L}$_{o, o}$ then so does $(X; Y^*)$. Hence, we only have to prove the converse.
Let $\e \in (0, 1)$ be given. Since $Y$ is uniformly convex, the pair $(Y;\mathbb{K})$ has the {\bf BPBop} with some $\tilde{\eta}(\e) > 0$ (see \cite[Theorem 2.1]{KL}). This means that if $y^*\in S_{Y^*}$ and  $y\in B_Y$ satisfy $|y^*(y)|>1-\tilde{\eta}(\e)$, then, there exists $z\in S_Y$ such that $|y^*(z)|=1$ and $\|y-z\|<\eps$. Fix $A \in \mathcal{L}(X,Y; \K)$ with $\|A\| = 1$ and take its associated operator $T_A \in S_{\mathcal{L}(X, Y^*)}$. Consider $\xi > 0$ to be such that $2\xi < \min \{\tilde{\eta}(\e), \e\}$ and set
	\begin{equation*}
		\eta(\e, A) :=  \min \{\xi, \eta'(\xi, T_A)\} > 0,	
	\end{equation*}
where $\eta'$ is the function in the definition of {\bf L}$_{o, o}$ for the pair $(X; Y^*)$. Let $(x_0, y_0) \in S_X \times S_Y$ be such that
	\begin{equation*}
		|A(x_0, y_0)| > 1 - \eta(\e, A).	
	\end{equation*}	
	Then, since
\begin{equation*}	
\|T_A(x_0)\|_{Y^*} \geq |T_A(x_0)(y_0)| = |A(x_0, y_0)| > 1 - \eta(\e, A) \geq 1 - \eta'(\xi, T_A),
\end{equation*}
there is $x_1 \in S_X$ such that
	\begin{equation*}
		\|T_A (x_1)\|_{Y^*} = 1 \ \ \ \mbox{and} \ \ \ \|x_1 - x_0\| < \xi < \e.	
	\end{equation*}
Now, since $T_A(x_1) \in S_{Y^*}$ and $y_0 \in S_Y$ satisfy
	\begin{eqnarray*}
		|[T_A(x_1)](y_0)| &\geq& |T_A(x_0)(y_0)| - |T_A(x_1 - x_0)(y_0)| \\
		&\geq& |A(x_0, y_0)| - \|x_1 - x_0\| \\
		&>& 1 - \eta(\e, A) - \xi \\
		&>& 1 - 2\xi > 1 - \tilde{\eta}(\e),
	\end{eqnarray*}
	there is $y_1 \in S_Y$ such that $|[T_A(x_1)](y_1)| = 1$ and $\|y_1 - y_0\| < \e$. Since
\begin{equation*}	
1 = |[T_A(x_1)](y_1)| = |A(x_1, y_1)|, \ \ \|x_1 - x_0\| < \e, \ \ \mbox{and} \ \ \|y_1 - y_0\| < \e,
\end{equation*}
we have proved that $(X,Y; \K)$ has the {\bf L}$_{o,o}$ for bilinear forms, as desired.	
\end{proof}
Denote by $X \tensor Y$ the projective tensor product of the Banach spaces $X$ and $Y$. Recall that the space $\mathcal{L}(X, Y; Z)$ is isometrically isomorphic to $\mathcal{L}( X  \tensor Y; Z)$ (see, for example, \cite[Theorem 2.9]{Rya}).
Recall also the following definition: a dual Banach space $X^*$ has the {\it $w^*$-Kadec-Klee property} if $\|x_{\alpha}-x\|\to 0$ whenever $\|x_{\alpha}\|\to \|x\|$ and $x_\alpha \xrightarrow[]{w^*}\, x$. If this holds for sequences, we say that $X^*$ has the \emph{sequential} {\it $w^*$-Kadec-Klee property}. For some background concerning these properties, see \cite{BorVan, HajTal}.  It is worth mentioning that if the unit ball $B_{X^*}$ is $w^*$-sequentially compact, then the {\it sequential $w^*$-Kadec-Klee property} implies the {\it $w^*$-Kadec-Klee property} on $X^*$ (see \cite[Proposition~1.4]{BorVan}). Now, we prove the desired result. 

\begin{theorem} \label{lp}  For $1<s<\infty$,  let $s'$ be the conjugate of $s$ (that is, $\frac{1}{s}+\frac{1}{s'}=1$).
\begin{itemize}	
\item[(a)]	If $2 < p , q < \infty$, then $(\ell_p, \ell_q; \mathbb{K})$ has the {\bf L}$_{p,p}$.
\item[(b)] If $1<p,q<\infty$, then $(\ell_p, \ell_q; \mathbb{K})$ has the {\bf L}$_{o,o}$ if and only if $pq>p+q$ (or, equivalently, $p>q'$).
\end{itemize}
\end{theorem}

\begin{proof}
	
{(a)} It is known that if $X^*$ has the $w^*$-Kadec-Klee property, then the pair $(X,\mathbb{K})$ has the {\bf L}$_{p,p}$ (see \cite[Proposition~2.6]{DKLM}). On the other hand, in \cite[Theorem~4]{DK} it was proved that if $1<r<2<s<\infty$, then $\mathcal{L}(\ell_s; \ell_r)=\mathcal{L}(\ell_s, \ell_{r'}; \K) = (\ell_s \hat{\tens}_{\pi} \ell_{r'})^*$ has the sequential $w^*$-uniform-Kadec-Klee property, which implies the sequential $w^*$-Kadec-Klee property. Indeed, since $\ell_s \hat{\tens}_{\pi} \ell_{r'}$ is reflexive (see, for instance, \cite[Corollary~4.24]{Rya}), then its unit dual ball is $w^*$-sequentially compact and, consequently, $(\ell_s \hat{\tens}_{\pi} \ell_{r'})^*$ has the $w^*$-Kadec-Klee property.
Hence, the pair $(\ell_p \hat{\tens}_{\pi} \ell_{q};\mathbb{K})$ has the  {\bf L}$_{p,p}$ for $2 < p, q < \infty$.

 For a given $\e > 0$ and a fixed norm-one point $(x,y)\in S_{\ell_p}\times S_{\ell_q}$, consider $\eta(\eps,x\tens y) > 0$ to be the function in the definition of  {\bf L}$_{p,p}$ for the pair $(\ell_p \hat{\tens}_{\pi} \ell_{q};\mathbb{K})$. Let $A \in \mathcal{L}(\ell_p,\ell_q;\mathbb{K})$ with $\|A\|=1$ be such that
\begin{equation*}
|A(x,y)| > 1 - \eta(\e, x \tens y).
\end{equation*}
Consider $\hat{A}$ to be the corresponding element in $S_{(\ell_p \hat{\tens}_{\pi} \ell_{q})^*}$ via the canonical isometry. Then, we have
\begin{equation*}
|\hat{A}(x\tens y)|=|A(x,y)|>1-\eta(\eps,x\tens y).
\end{equation*}
Since the pair $(\ell_p \hat{\tens}_{\pi} \ell_{q};\mathbb{K})$ has the {\bf L}$_{p,p}$ with $\eta(\e, x \tens y) > 0$, there exists $\hat{B}\in S_{(\ell_p \hat{\tens}_{\pi} \ell_{q})^*}$ such that
\begin{equation*}
 |\hat{B}(x\tens y)|=1 \ \ \ \mbox{and} \ \ \ \|\hat{B}-\hat{A}\|<\e.
\end{equation*}
Now we take $B\in S_{\mathcal{L}(\ell_p,\ell_q; \K)}$, the corresponding element to $\hat{B}$ via the canonical isometry. Then, $|B(x, y)| = |\hat{B}(x\tens y)|= 1$ and $\|B - A\| = \|\hat{B}-\hat{A}\|<\e$. This proves (a).

\vspace{0.2cm}
\noindent
{(b)} Let $1<p,q<\infty$. By Lemma~\ref{uv}, $(\ell_p, \ell_q; \mathbb{K})$ has the {\bf L}$_{o,o}$ if and only if $(\ell_p; \ell_{q'})$ has the {\bf L}$_{o,o}$ and, in virtue of \cite[Theorem 2.21]{D}, this happens if and only if $p>q'$.
\end{proof}

Note that inside the proof of Theorem \ref{lp}, we have proved that the pair $(\ell_p \hat{\tens}_{\pi} \ell_{q};\mathbb{K})$ has the {\bf L}$_{p,p}$ for $2 < p, q < \infty$. This yields to the following consequence.
\begin{cor} For $p,q\geq 1$

(a) if $2 <p, q < \infty$, then the norm of $\ell_p \hat{\tens}_{\pi} \ell_{q}$ is SSD.	

(b)  if $p^{-1}+q^{-1}\geq 1$  or one of $p$ and $q$ is $1$ or $\infty$, then the norm of $\ell_p \hat{\tens}_{\pi} \ell_{q}$ is not SSD.
 \end{cor}
\begin{proof}
As we already mentioned, item (a) follows from the proof of Theorem~\ref{lp} and  \cite[Theorem 2.3]{DKLM}. To prove (b), note that if $p^{-1}+q^{-1}\geq 1$, then the main diagonal $\mathcal{D}=\overline{\mbox{span} \{e_n\otimes e_n:\,\, n\in \mathbb{N}\}}$ is one-complemented in $\ell_p \hat{\tens}_{\pi} \ell_{q}$ and isometrically isomorphic to $\ell_1$ (see, for instance, \cite[Theorem~1.3]{AF}). Hence, if the norm of $\ell_p \hat{\tens}_{\pi} \ell_{q}$ were SSD, by \cite[Theorem 2.3]{DKLM} we would have that $(\ell_p \hat{\tens}_{\pi} \ell_{q},\K)$ has the {\bf L}$_{p,p}$ and,
by \cite[Proposition~4.4~(b)]{DKLM}, $(\ell_1, \K)$ would have the {\bf L}$_{p,p}$, which gives the desired contradiction. Suppose now that $p$ or $q$ take the value 1 or $\infty$. As we showed in the proof of  Theorem~\ref{lp}, if $\ell_p \hat{\tens}_{\pi} \ell_{q}$ were SSD then $(\ell_p, \ell_q; \mathbb{K})$ would have the {\bf L}$_{p,p}$ for bilinear forms, which is not possible by Proposition~\ref{bil3}.(c) since neither $\ell_1$ nor $\ell_\infty$ are not SSD.
\end{proof}
In the proof of Theorem~\ref{lp} we showed that if the pair $(\ell_p \hat{\tens}_{\pi} \ell_{q};\mathbb{K})$ has the {\bf L}$_{p,p}$ (or, equivalently, $\ell_p \hat{\tens}_{\pi} \ell_{q}$ is SSD) then $(\ell_p, \ell_q; \mathbb{K})$ has the {\bf L}$_{p,p}$ for bilinear forms. However, it is worth to remark that the converse is not true. For instance, $(\ell_2, \ell_2; \mathbb{K})$ has the {\bf L}$_{p,p}$ for bilinear forms (moreover, it has the {\bf BPBpp} by \cite[Corollary~3.2]{DKL}) but $\ell_2 \hat{\tens}_{\pi} \ell_{2}$ is not SSD.

We finish this section with some remarks and open questions.
\begin{rem}\label{remark lp}
{\bf (a)} Since the uniform properties imply the local properties, when trying to prove that $(X_1,\dots, X_N;\K)$ has the {\bf L}$_{p,p}$ (respectively, {\bf L}$_{o,o}$) for some Banach spaces $X_1,\dots, X_N$, it is natural to ask first if $(X_1,\dots, X_N;\K)$ has (or not) the {\bf BPBpp} (respectively,  {\bf BPBop}). Taking into account Theorem~\ref{lp} we must say that, to the best of our knowledge, it is not known whether $(\ell_p, \ell_q; \mathbb{K})$ has the {\bf BPBpp} when $2 < p , q < \infty$. On the other hand, by Corollary~\ref{failBPBop}, $(\ell_p, \ell_q; \mathbb{K})$ fails the {\bf BPBop} for every $1<p, q<\infty$.

{\bf (b)} By Proposition~\ref{bil3} we know that if $(X_1,\dots, X_N;\K)$ has the {\bf L}$_{p,p}$ (respectively, {\bf L}$_{o,o}$), then so does $(X_i;\mathbb{K})$ for $1\leq i\leq N$. Hence, we may ask if $(X_1,\dots, X_N;\K)$ has one of the mentioned properties whenever the pairs $(X_i;\mathbb{K})$ does. In that sense, note that $(\ell_p; \K)$ and $(\ell_q;\K)$ both have the {\bf L}$_{o,o}$ for every $1<p, q<\infty$ (since $\ell_p$ and $\ell_q$ are both reflexive and $\ell_p^*, \ell_q^*$ are both SSD) but, in virtue of Theorem~\ref{lp}~(b), there are $p, q$ such that $(\ell_p, \ell_q; \mathbb{K})$ fails the {\bf L}$_{o,o}$. We also have that the pairs $(\ell_p;\K)$ and $(\ell_q; \K)$ have the {\bf L}$_{p,p}$ for every $1<p, q<\infty$ but we do not know if there is some $1<p,q<\infty$ such that $(\ell_p, \ell_q; \mathbb{K})$ fails the {\bf L}$_{p,p}$ for bilinear forms.

\end{rem}

\section{Local AHSP}

Our main aim in this section is to give a characterization for the Banach space $Y$ in such a way that $(\ell_1^N, Y; \K)$ satisfies the {\bf L}$_{p,p}$. Indeed, we prove that the norm of a Banach space $Y$ is SSD if and only if $(\ell_1^N, Y; \K)$ has the {\bf L}$_{p, p}$ for bilinear forms. To do so, we get a characterization of SSD that is motivated by the {\it approximate hyperplane series property} (AHSP, for short), which was defined for the first time in \cite{AAGM}. Before giving our characterization, we recall the definition and important results concerning this property.

\begin{definition}[\cite{AAGM}]\label{ahsp} A Banach space $Y$ has the AHSP if for every $\e > 0$, there is $\eta(\e) > 0$ such that given a sequence $(y_k) \subset S_Y$ and a convex series $\sum_{k=1}^{\infty} \alpha_k$ such that
\begin{equation*}
\left\| \sum_{k=1}^{\infty} \alpha_k y_k \right\| > 1 - \eta(\e),
\end{equation*}
there exist $A \subset \N$, $y^* \in S_{Y^*}$, and $\{z_j: j \in A\}$ satisfying the following conditions:
\begin{equation*}
\sum_{k \in A} \alpha_k > 1 - \e, \ \ \ \|z_k - y_k\| < \e, \ \ \ \mbox{and} \ \ \ y^*(z_k) = 1 \text{~for every~} k \in A.
\end{equation*}
\end{definition}

Finite dimensional, uniformly convex and lush spaces are known examples of Banach spaces satisfying the AHSP (see \cite[Propositions 3.5, 3.8]{AAGM} and \cite[Theorem 7]{CK}, respectively). More specifically, $L_p(\mu)$-spaces for arbitrary $1 \leq p \leq \infty$ and $C(K)$-spaces for a compact Hausdorff $K$ are concrete examples of such a Banach spaces. This property was defined in \cite{AAGM} in order to give a characterization for the Banach spaces $Y$ such that the pair $(\ell_1; Y)$ has the {\bf BPBp} for operators. Here, we are interested to get a local version of AHSP which is related with the {\bf L}$_{p,p}$ for bilinear mappings (see \cite[Definition~3.1]{ABGM} and \cite[Section 3]{DGKLM} for AHSP for bilinear mappings). It turns out that this local version of AHSP is equivalent to SSD of the norm.


\begin{prop} \label{SSDe} Let $Y$ be a Banach space. For any $N \in \N$, the following are equivalent.
\begin{itemize}
\item[(a)] The norm of $Y$ is SSD.
\vspace{0.1cm}
\item[(b)] Given $\e > 0$, a nonempty set $A \subset \{1, \ldots, N \}$, $(\alpha_j)_{j \in A}$ with $\alpha_j > 0$ for all $j \in A$ and $\sum_{j \in A} \alpha_j = 1$, and $y \in S_Y$,  there is $\eta = \eta(\e, (\alpha_j)_{j \in A}, y) > 0$ such that  whenever $(y_j^*)_{j \in A} \subset S_{Y^*}$ satisfies
	\begin{equation*}
	\re \sum_{j \in A} \alpha_j y_j^*(y) > 1 - \eta,	
	\end{equation*}
there is $(z_j^*)_{j \in A} \subset S_{Y^*}$ such that
\begin{equation*}
z_j^*(y) = 1 \ \ \ \mbox{and} \ \ \ \|z_j^* - y_j^*\| < \e,	
\end{equation*}	
for all $j \in A$.
\end{itemize}
\end{prop}

\begin{proof} (b) implies (a) by considering a singleton $A$ and recalling that $Y$ is SSD if and only if $(Y,\K)$ has the {\bf L}$_{p,p}$. Now assume that the norm of $Y$ is SSD or, equivalently, that the pair $(Y; \K)$ has the {\bf L}$_{p,p}$ with some function $\eta'$.  Fix $\e > 0$, a nonempty set $A \subset \{1, \ldots, N \}$, $(\alpha_j)_{j \in A}$ with $\alpha_j > 0$ for all $j \in A$ and $\sum_{j \in A} \alpha_j = 1$, and $y \in S_Y$. Set $\alpha := \min_{j\in A}\alpha_j$ and
\begin{equation*}
\eta(\e, (\alpha_j)_{j \in A}, y) := \alpha \eta'\left( \frac{\e^2}{16},y \right) > 0.
\end{equation*}
Note that we may assume that $\eta'(\e, y) \leq \e$ for every $\e > 0$. Let $(y_j^*)_{j \in A} \subset S_{Y^*}$ be such that
	\begin{equation*}
		\re \sum_{j \in A} \alpha_j y_j^*(y) > 1 - \eta(\e, (\alpha_j)_{j \in A}, y).	
	\end{equation*}
Then, for each $k \in A$, we have
\begin{equation*}
1 - \alpha_k \eta'(\e, y) \leq 1 - \alpha \eta'(\e, y)
= 1 - \eta(\e, (\alpha_j)_{j \in A}, y) <\re \sum_{j \in A} \alpha_j y_j^*(y)
\leq \re \alpha_k y_k^*(y) + (1 - \alpha_k).
\end{equation*}
So,
\begin{equation*}
\re y_k^*(y) > 1 - \eta'\left( \frac{\e^2}{16},y \right) \ \  \mbox{for each $k \in A$}.
\end{equation*}
Since $(Y; \K)$ has the {\bf L}$_{p,p}$ with $\eta'$, for each $k \in A$, there is $\tilde{z}_k^* \in S_{Y^*}$ such that
\begin{equation*}
|\tilde{z}_k^*(y)| = 1 \ \ \ \mbox{and} \ \ \  \|\tilde{z}_k^* - y_k^*\| < \frac{\e^2}{16}.
\end{equation*}
For each $k \in A$, write $\tilde{z}_k^*(y) = e^{i \theta_k} |\tilde{z}_k^*(y)| =  e^{i \theta_k}$. Then,
\begin{equation*}
\|e^{-i\theta_k} \tilde{z}_k^* - y_k^*\| \leq |1 - e^{i \theta_k}| + \|\tilde{z}_k^* - y_k^*\|
\end{equation*}
for all $k \in A$. Now, note that whenever $k \in A$, we have
\begin{equation*}
\re e^{i \theta_k} = \re \tilde{z}_k^*(y) \geq \re y_k^*(y) - \|\tilde{z}_k^* - y_k^*\| > 1 - \eta'\left( \frac{\e^2}{16},y \right) - \frac{\e^2}{16} > 1 - \frac{\e^2}{8}.	
\end{equation*}
So,
\begin{eqnarray*}
|1 - e^{i \theta_k}|^2 &=& (\re e^{i \theta_k} - 1)^2 + (\im e^{i \theta_k} )^2 \\
&=& 1 - 2 \re e^{i \theta_k} +(\re e^{i \theta_k})^2 + (\im e^{i \theta_k} )^2 \\
&=& 2 (1 - \re e^{i \theta_k}) < \frac{\e^2}{4},
\end{eqnarray*}
which implies $|1 - e^{i \theta_k}| < \frac{\e}{2}$ for every $k \in A$. Then, for each $k \in A$, we have
\begin{equation*}
\|e^{-i\theta_k}\tilde{z}_k^* - y_k^*\| < \frac{\e}{2} + \frac{\e^2}{16} < \e.
\end{equation*}
Setting $z_k^* := e^{-i \theta_k} \tilde{z}_k^*$ for each $k \in A$, we have that $z_k^*(y) = 1$ and $\|z_k^* - y_k^*\| < \e$, which proves that (a) implies (b).
\end{proof}

Note that part (b) of Proposition \ref{SSDe} is a kind of local version of AHSP for the Bishop-Phelps-Bollob\'as point property since we do not move the initial point and also the $\eta$ in its definition depends not only on a positive $\e$ but also on a finite convex series and on a norm-one point. Observe that, by a simple change of parameters, we can take $(y_j^*)_{j \in A}$ in $B_{Y^*}$ instead of $S_{Y^*}$ in item (b) and we are using this fact without any explicit reference in the next theorem, where we prove the promised characterization of property {\bf L}$_{p,p}$ for $(\ell_1^N, Y; \K)$.

\begin{theorem} \label{AHSPbil1} Let $Y$ be a Banach space and $N \in \N$. Then,  $(\ell_1^N, Y;\K)$ has the {\bf L}$_{p, p}$ if and only if the norm of $Y$ is SSD.	
\end{theorem}

\begin{proof}
If $(\ell_1^N, Y; \K)$ has the {\bf L}$_{p, p}$ for bilinear forms then, by Proposition~\ref{bil3} (c), the pair $(Y; \mathbb{K})$ has the {\bf L}$_{p, p}$, which is equivalent to say that the norm of $Y$ is SSD. Suppose now that the norm of $Y$ is SSD. Let $\e > 0$ and $(x, y) \in S_{\ell_1^N} \times S_Y$ be given. We write $x = (x_1, \ldots, x_N)$ and assume that $x_j \geq 0$ for all $j = 1, \ldots, N$ by composing it with an isometry if necessary. Let $A = \{j \in \{1, \ldots, N\}: x_j \not= 0 \}$. Then $\|x\|_1 = \sum_{j \in A} x_j = 1$. Consider $(x_j)_{j \in A}$ and by Proposition \ref{SSDe}, we may set
	\begin{equation*}
	\eta(\e, (x, y)) := \eta(\e, (x_j)_{j \in A}, y) > 0.
	\end{equation*}
Let $A \in \mathcal{L}(\ell_1^N, Y; \K)$ with $\|A\| = 1$ be such that
\begin{equation*}
|A(x, y)| > 1 - \eta(\e, (x, y)).	
\end{equation*}	
By rotating $A$, if necessary, we may assume that $\re A (x, y) > 1 - \eta(\e, (x, y))$.	So,
\begin{equation*}
\re A(x, y) = \re \sum_{j \in A} x_j A(e_j, y) > 1 - \eta(\e, (x, y)).	
\end{equation*}
Define $y_j^* (y) := A(e_j, y)$ for every $y \in Y$. Since $\|A\| = 1$, we have that $(y_j^*)_{j \in A} \subset B_{Y^*}$ and
\begin{equation*}
\re \sum_{j \in A} x_j y_j^*(y) > 1 - \eta(\e, (x, y)) = 1 - \eta(\e, (x_j)_{j \in A}, y).	
\end{equation*}
By Proposition \ref{SSDe}, there is $(z_j^*)_{j \in A} \subset S_{Y^*}$ such that $z_j^*(y) = 1$ and $\|z_j^* - y_j^*\| < \e$, for all $j \in A$. Now, define $B: \ell_1^N \times Y \longrightarrow \K$ by
\begin{equation*}
B(u, v) := \sum_{j \in A} u_j z_j^*(v) + \sum_{j \not\in A} u_j A(e_j, v),	
\end{equation*}
for $u = (u_j)_{j=1}^N \in \ell_1^N$ and $v \in Y$. So, $\|B\| \leq 1$ and
\begin{equation*}
|B(x, y)| = \left| \sum_{j \in A} x_j B(e_j, y) \right| = \left| \sum_{j \in A} x_j z_j^*(y) \right| = \sum_{j \in A} x_j = 1.
\end{equation*}
Then $\|B\| = |B(x, y)| = 1$. Also, for every $(u, v) \in S_{\ell_1^N} \times S_Y$, we have
\begin{equation*}
|B(u, v) - A(u, v)| = \left|\sum_{j \in A} u_j (z_j^* - y_j^*)(v) \right| < \e \sum_{j \in A} u_j \|v\| \leq \e.	
\end{equation*}
Therefore $\|B - A\| < \e$, and this shows that $(\ell_1^N, Y; \K)$ has the {\bf L}$_{p, p}$ for bilinear forms.
\end{proof}

\begin{rem}\label{counter}
We can use Theorem \ref{AHSPbil1} to show that Proposition \ref{bil3}.(b) does not hold for {\bf L}$_{p, p}$. Consider a dual space $Y^*$ which is isomorphic to $\ell_1$ and its norm is locally uniformly rotund (and then strictly convex) \cite{Kadec}. Then, the norm of $Y$ is Fr\'echet differentiable (see, for example, \cite[Fact 8.18]{FHHMPZ}), and so it is SSD. By Theorem \ref{AHSPbil1}, we have that  $(\ell_1^N,Y;\mathbb{K})$ has the  {\bf L}$_{p, p}$. Suppose by contradiction that the pair $(\ell_1^N; Y^*)$ has the {\bf L}$_{p, p}$ for operators. Since {\bf L}$_{p, p}$ is stable under one-complemented subspaces on the domain (see \cite[Proposition 4.4]{DKLM}), the pair $(\ell_1^2; Y^*)$ also has the {\bf L}$_{p, p}$. Since $Y^*$ is strictly convex, by using \cite[Proposition 3.2]{DKLM} we get that $Y^*$ should be uniformly convex, which is not possible. So, $Y$ is the desired counterexample.
\end{rem}
Analogously to the bilinear case, we obtain a characterization of those Banach spaces $Y$ such that the pair $(\ell_1^N; Y)$ has the {\bf L}$_{p, p}$ for operators for an arbitrary $N \in \N$. Since the proof is quite similar to Theorem \ref{AHSPbil1}, we omit the details.
\begin{prop} \label{AHSP1} Let $Y$ be a Banach space. The pair $(\ell_1^N; Y)$ has the {\bf L}$_{p, p}$ for operators if and only if given $\e>0$, a nonempty set $A \subset \{1, \ldots, N \}$ and $(\alpha_j)_{j \in A}$ with $\alpha_j > 0$ for all $j \in A$ such that $\sum_{j \in A} \alpha_j = 1$,  there is $\eta = \eta(\e, (\alpha_j)_{j \in A}) > 0$ such that whenever $(y_j)_{j \in A} \subset S_Y$ satisfies
	\begin{equation*}
		\left\| \sum_{j \in A} \alpha_j y_j \right\| > 1 - \eta,	
	\end{equation*}
	there are $z^* \in S_{Y^*}$ and $(z_j)_{j \in A} \subset S_Y$ such that
	\begin{equation*}
		z^*(z_j) = 1 \ \ \ \mbox{and} \ \ \ \|z_j - y_j\| < \e,	
	\end{equation*}	
	for all $j \in A$.		
\end{prop}

It turns out that the AHSP (see Definition~\ref{ahsp}) implies the characterization of Proposition~\ref{AHSP1}, as we show in the following theorem. This provide new examples of spaces $Y$ such that $(\ell_1^N; Y)$ has the {\bf L}$_{p, p}$ for linear operators.
In particular, if $Y$ is a uniformly convex Banach space, then the pair $(\ell_1^N; Y)$ has the {\bf L}$_{p, p}$, a result that was already proved in \cite[Proposition 2.10]{DKLM}.

\begin{theorem} Let $Y$ be a Banach space and $N \in \N$. If $Y$ has AHSP, then $(\ell_1^N; Y)$ has the {\bf L}$_{p, p}$.
\end{theorem}
\begin{proof}
Assume that $Y$ has AHSP with a function $\eta$ and fix  $\e>0$, a nonempty $A \subset \{1, \ldots, N \}$ and a sequence of positive numbers $(\alpha_j)_{j \in A}$ with $\sum_{j \in A} \alpha_j = 1$. Take $0<\lambda<\min\big\{\e,\min\{\alpha_j~:~j\in A\}\big\}$ and assume that a sequence of vectors $(y_j)_{j\in A}\subset S_Y$ satisfies
\begin{equation*}
\left\| \sum_{j \in A} \alpha_j y_j \right\| > 1 - \eta(\lambda).
\end{equation*}
By the definition of AHSP, there are $B\subset A$, $\{z_k : k\in B\}\subset S_X$ and $z^*\in S_{Y^*}$ such that
\begin{equation*}
\sum_{k\in B} \alpha_k>1-\lambda, \ \ \  \|z_k-x_k\|<\lambda, \ \ \  \mbox{and} \ \ \ x^*(z_k)=1
\end{equation*}
for all $k\in B$. Since $\lambda<\min\{\alpha_j~:~j\in A\}$, $A=B$. By Proposition \ref{AHSP1},  $(\ell_1^N; Y)$ has the {\bf L}$_{p, p}$.
\end{proof}

\begin{cor} Let $Y$ be a Banach space and $N \in \N$. If $Y$ is a
	\begin{itemize}
		\item[(a)] finite dimensional space or
		\item[(b)] uniformly convex space or
		\item[(c)] lush space,
\end{itemize}
then the pair $(\ell_1^N; Y)$ has the {\bf L}$_{p, p}$ for operators.
\end{cor}

We also get a result about uniformly strongly exposed family. We say that a family $\{y_{\alpha}\}_{\alpha} \subset S_Y$ is {\it uniformly strongly exposed} with respect to a family $\left\{f_\alpha\right\}_\alpha\subset S_{Y^*}$, if there
is a function $\eps \in (0,1) \mapsto \delta(\eps)>0$ such that $f_\alpha(y_\alpha) = 1$ for all $\alpha$ and $\re f_\alpha(y)> 1-\delta(\eps)$ implies $\|y-y_\alpha\|<\eps$ whenever $y\in B_X$.

\begin{prop} \label{AHSP4} Let $Y$ be a Banach space and let $\{y_{\alpha}\}_{\alpha} \subset S_Y$ be a uniformly strongly exposed family with corresponding functionals $\{f_{\alpha}\}_{\alpha} \subset S_{Y^*}$. If $\{f_{\alpha}\}_{\alpha}$ is a norming subset for $Y$, then the pair $(\ell_1^N; Y)$ has the {\bf L}$_{p, p}$.	
\end{prop}

\begin{proof} Let $N \in \N$ and $A \subset \{1, \ldots, N\}$ be a nonempty finite subset. Let $\e > 0$ and suppose there is $\alpha := (\alpha_{j})_{j \in A}$ such that $\alpha_j > 0$ for all $j \in A$ and $\sum_{j \in A} \alpha_j = 1$. Set $K_{\alpha} := \min \{ \alpha_j: j \in A \}$ and define
	\begin{equation*}
	\eta = \eta(\e, (\alpha_j)_{j \in A}) := K_{\alpha} \delta \left( \frac{\e}{2} \right) > 0,	
	\end{equation*}
where $\e \mapsto \delta(\e)$ is the function for the family $\{y_{\alpha}\}_{\alpha}$. Let $(y_j)_{j \in A} \subset S_Y$ be such that
\begin{equation*}
\left\| \sum_{j \in A} \alpha_j y_j \right\| > 1 - \eta.	
\end{equation*}	
Since $\{f_{\alpha}\}$ is norming for $Y$, we may take $\alpha_0$ to be such that
\begin{equation*}
\sum_{j \in A} \alpha_j \re f_{\alpha_0} (y_j)  = \re f_{\alpha_0} \left( \sum_{j \in A} \alpha_j y_j \right) > 1 - \eta.	
\end{equation*}	
Then, for each $i \in A$, we have
\begin{equation*}
1 - K_{\alpha} \delta \left(\frac{\e}{2}\right) <  \sum_{j \in A} \alpha_j \re f_{\alpha_0} (y_j)
\leq \sum_{j \in A \setminus \{i\}} \alpha_j + \alpha_i \re f_{\alpha_0} (y_i)
=1 - \alpha_1 + \alpha_i \re f_{\alpha_0}(y_i).
\end{equation*}
Therefore, for each $i \in A$, we have
\begin{equation*}
1 - \re f_{\alpha_0} (y_i) < \frac{K_{\alpha}}{\alpha_i} \delta \left(\frac{\e}{2} \right) \leq \delta \left(\frac{\e}{2} \right),	
\end{equation*}
which implies that $\re f_{\alpha_0} (y_i) > 1 - \left(\frac{\e}{2} \right)$ for every $i \in A$. So, we have that $\|y_i - y_{\alpha_0}\| < \frac{\e}{2}$ for all $i \in A$. Thus, for every $n, m \in A$, we have that
\begin{equation*}
\|y_n - y_m\| \leq \|y_n - y_{\alpha_0}\| + \|y_{\alpha_0} - y_m\| < \e.	
\end{equation*}
Since $A \not= \emptyset$, choose $n_0 \in A$ and set $z_j := y_{n_0}$ for all $j \in A$. Then, $z_j \in S_Y$ and $\|z_j - y_j\| = \|y_{n_0} - y_j\| < \e$ for all $j \in A$. Finally, take $y^* \in S_{Y^*}$ to satisfy $y^*(z_j) = y^*(y_{n_0}) = 1$. By Proposition \ref{AHSP1}, $(\ell_1^N; Y)$ has the {\bf L}$_{p, p}$.
\end{proof}










	

\end{document}